\documentclass[reqno]{amsart}
\usepackage[utf8]{inputenc}
\usepackage{amsmath,amsthm,amssymb}
\usepackage{hyperref}
\date{\today}
\usepackage{paralist}
\theoremstyle{definition}
\newtheorem{theorem}{Theorem}
\newtheorem{remark}[theorem]{Remark}
\newtheorem{lemma}[theorem]{Lemma}
\newtheorem{corollary}[theorem]{Corollary}

\newtheorem{proposition}[theorem]{Proposition}
\numberwithin{equation}{section}

\theoremstyle{definition}
\theoremstyle{definition}
\newtheorem{definition}[theorem]{Definition}

\usepackage{hyperref}
\usepackage{graphicx}
\usepackage{todonotes}
\usepackage{comment}
\begin{document}

\newcommand{\Pos}{Pos}
\newcommand{\Hess}{\operatorname{Hess}}
\newcommand{\Id}{\operatorname{Id}}
\newcommand{\Sym}{\operatorname{Sym}}
\newcommand{\Hom}{\operatorname{Hom}}

\newcommand{\argmin}{\operatorname{argmin}}
\title[Differentiability of the argmin function]{Differentiability of the argmin function and a minimum principle for semiconcave subsolutions}
\author{Julius  Ross and David Witt Nystr\"om}

\begin{abstract}
Suppose  $f(x,y) + \frac{\kappa}{2} \|x\|^2 - \frac{\sigma}{2}\|y\|^2$ is convex where $\kappa\ge 0, \sigma>0$, and the argmin function $\gamma(x) = \{ \gamma : \inf_y f(x,y) = f(x,\gamma)\}$ exists and is single valued.  We will prove $\gamma$ is  differentiable almost everywhere.  As an application we deduce a minimum principle for certain semiconcave subsolutions.
\end{abstract}
\maketitle

%\section{$F$-subharmonic functions}
%\subsection{Notation}
%We let $M_n$ be the set of real $n\times n$ matrices and $S_n$ be the subset of symmetric matrices.  Let
%$\Pos_n\subset S_n$ be the set of $n\times n$ symmetric real semipositive matrices, that is 
%$$\Pos_n = \{ A \in S_n : v^t A v\ge 0  \mbox{ for all } v\in \mathbb R^n\}.$$
\renewcommand{\Pos}{\operatorname{Pos}}
\newcommand{\calPos}{\mathcal P}
\newcommand{\Int}{\operatorname{Int}}
\setcounter{tocdepth}{1}
%\tableofcontents

\section{Introduction}\label{sec:introduction}

The first part of this paper is a proof of the following elementary statement about regularity of certain argmin functions. 

\begin{theorem}\label{thm:intoargmin1} \ 
Suppose $f:\mathbb R^{n}\times \mathbb R^{m}\to \mathbb R$ is such that
\begin{enumerate}
    \item There are $\kappa\ge 0$ and $\sigma>0$ so  $$f(x,y) + \frac{\kappa}{2} \|x\|^2-\frac{\sigma}{2}\|y\|^2\text{ is convex.}$$
    \item For each $x\in \mathbb R^{n}$ there is a unique $\gamma(x)$ such that 
    $$\inf_y f(x,y) = f(x,\gamma(x)).$$
\end{enumerate}
Then the function $\gamma$ is differentiable almost everywhere.
%calm at almost all points $x_0\in \mathbb R^n$.  That is, for almost all $x_0$ there are $C$ and $\delta>0$ such that
%$$ \|\gamma(x) - \gamma(x_0)\| \le C \| x-x_0\|\text{ for } \|x-x_0\|<\delta.$$
%Equivalently (by Stepanov's Theorem) $\gamma$ is differentiable almost everywhere.
\end{theorem}

Our motivation for this is the following.  Following Harvey-Lawson \cite{HL_Dirichletduality,HL_Dirichletdualitymanifolds}, by a \emph{(primitive) subequation} on an open $X\subset \mathbb R^n$ we mean a subset $F\subset J^2(X)$ of the space of $2$-jets on $X$ with certain properties.   Given such an $F$ and a $\mathcal C^2$ function $f$, we say that $f$ is \emph{$F$-subharmonic} if every $2$-jet of $f$ lies in $F$.  Moreover, using the so-called viscosity technique it is possible to extend the notion of $F$-subharmonicity to any upper-semicontinuous function (details and precise definitions will be given in  \S\ref{sec:Fsubharmonics}).  

In our previous work \cite{MP1} we introduced a notion of ``product subequation" $F\#\calPos$ on $X\times \mathbb R^m$ and show (under suitable hypothesis) that if $F$ is convex and $f$ is $F\#\calPos$-subharmonic then its marginal function
$$ g(x) := \inf_y f(x,y)$$
is $F$-subharmonic.  This statement generalises the classical statement that the marginal function of a convex function is again convex.  We will use Theorem \ref{thm:intoargmin1} to prove a similar minimum principle that does not require $F$ to be convex:
%\begin{definition}[The Pucci Extremal Operators]
%\end{definition}

\begin{theorem}\label{thm:introminimum}
Let $X\subset \mathbb R^n$ be open and $F\subset J^2(X)$ be a constant-coefficient primitive subequation that depends only on the Hessian part.  Suppose 
$$f:X\times \mathbb R^m\to \mathbb R$$
is locally semiconcave, bounded from below, and $F\#\calPos$-subharmonic.  Then the marginal function
$$ g(x) = \inf_{y} f(x,y)$$
is $F$-subharmonic on $X$.
\end{theorem}\vspace{1mm}
\noindent A few remarks are in order.\vspace{2mm}

 \begin{asparaenum}\setlength{\itemsep}{3mm}
\item The semiconcavity assumption on $f$ is rather unnatural, since one would expect a subsolution to have some kind of convexity rather than concavity, but it captures what we are able to prove.  Observe that $f$ is certainly locally semiconcave if it is $\mathcal C^{1,1}_{loc}$.   
\item The assumption that $f$ is $F\#\calPos$-subharmonic implies that for each $x$ the function $y\mapsto f(x,y)$ is convex.    This along with the semiconcavity assumption implies that $y\mapsto f(x,y)$ is $\mathcal C^{1,1}_{loc}$.
\item Theorem \ref{thm:introminimum} can be  proved rather easily when $f$ is $\mathcal C^2$ (see \cite[Prop. 7.5]{MP1} for a stronger statement).  To do so, we first approximate $f$ by adding a small multiple of the function $(x,y)\mapsto \|y\|^2$, so there is no loss in assuming $f$ is strictly convex in $y$ and that for each fixed $x$ the function $y\mapsto f(x,y)$ attains its unique minimum at some point $\gamma(x)$.  Said another way, $\gamma(x)$ is the unique point such that
$$ \frac{\partial f}{\partial y}|_{(x,\gamma(x))} = 0.$$
If we assume $f$ is $\mathcal C^2$ we can then:
\begin{enumerate}[(a)]
\item Use the implicit function theorem to deduce that $\gamma$ is $\mathcal C^1$.
\item Use the chain rule to compute the Hessian of $g$ at a point $x$ in terms of the Hessian of $f$ at the point $(x,\gamma(x))$ and the derivative of $\gamma$ at $x$.
\end{enumerate}
The combination of (b) and assumption that $f$ is $F\#\calPos$-subharmonic yields that $g$ is $F$-subharmonic as claimed.
\item If we assume furthermore that $F$ is convex, then using smooth mollification to approximate any upper-semicontinuous $F\#\calPos$-subharmonic function by those that are $\mathcal C^2$, we can deduce a much more general minimum principle -- this is the approach taken in \cite{MP1}.
\item If instead we assume that $f$ is merely $\mathcal C^{1,1}_{loc}$ then it is of course twice differentiable almost everywhere.  However it may well be that $f$ is not twice differentiable at any point of the form $(x,\gamma(x))$ so part (b) of the above argument does not apply.\medskip
\end{asparaenum}

To prove Theorem \ref{thm:introminimum} we will first use a partial-sup convolution to approximate $f$ by $F\#\calPos$-subharmonic functions $f_{\epsilon}$ such that
$$ f_{\epsilon}(x,y) + \frac{1}{2\epsilon} \|x\|^2 - \frac{\epsilon}{2}\|y\|^2 \text{ is convex}.$$
In particular for fixed $x$ the function $y\mapsto f_{\epsilon}(x,y)$ is strongly convex, and we will further arrange so the argmin of ${f_{\epsilon}}$ is a well-defined single-valued function $\gamma$.  Having done so we can apply Theorem \ref{thm:intoargmin1}  to deduce that $\gamma$ is differentiable almost everywhere, which will act in lieu of the implicit function argument used in (a).   From this one can prove, essentially from the definition, that at almost every point $x$ the Hessian of $g$ is contained in $F$.  As $g$ is semiconvex, this is known by the Almost-Everywhere Theorem of Harvey-Lawson \cite{HL_AE} to be enough to conclude that $g$ is $F$-subharmonic.

\subsection*{Comparison with other work: }
The authors do not have sufficient expertise to properly survey all previously known regularity results that are related to Theorem \ref{thm:intoargmin1}.  Suffice to say there has been much interest in studying regularity of marginal functions (by which we mean functions of the form $\inf_y f(x,y)$ or $\sup_y f(x,y)$  for some function $f$ which also go under the name ``performance function") due to its relevance for optimization problems (see for instance \cite{Aubin1,Quasidifferentiabilityandrelatedtopics,non-smooth,Globaloptimzationinaction} and the references therein).   For example, various regularity properties of marginal functions have been shown when $f$ has some convexity property (see for example \cite[Theorems 23.4 and 24.5]{Rockafellar_Convex_analysis}) and without this convexity hypothesis (e.g.  \cite{Clarke_generalizedgradients,Penot1982a,Penot04,Penot04b,ward1} to list just a few). 

Much less appears to have been written about regularity of the argmin function itself.    We remark that in general the argmin function will be multi-valued, and so regularity must be phrased in terms of set-valued functions \cite{Clarke_generalizedgradients}.   The only previous such results we have found relate to continuity rather than differentiability (for example \cite[Theorem 2.10]{wets_Lipschitzcontinuity},  which is taken from \cite[Theorems 1.17 and 7.41]{Rockafellar_Wets_VariationalAnalysis},  gives conditions under which the argmin function is outer semicontinuous). 

Regarding the minimum principle, the fact that the marginal function of a convex function is again convex is a basic property in convex analysis.  In the complex case this has an analog for plurisubharmonic functions due to Kiselman \cite{KiselmanInvent,Kiselman}.  Both convexity and plurisubharmonicity are massively generalized through the notion of $F$-subharmonic functions which uses the viscosity technique that arose in the study of fully non-linear degenerate second-order differential equations (in particular the work of Caffarelli--Nirenberg--Spruck \cite{CaffarelliNirenbergSpruckIII} and Lions--Crandall--Ishii \cite{CILUserGuide}, who often refer to such functions as \emph{subsolutions}).    

Our motivation for introducing the product $F\#\calPos$ came from a desire to generalise this minimum principle to general subequations, which we do in \cite{MP1} under the assumption that $F$ is convex.  As discussed above, this assumption is needed only to be able to approximate $F\#\calPos$-subharmonic functions by smooth ones, and thus suggests that it is a facet of the proof rather than an essential requirement.  Theorem \ref{thm:introminimum} is, as far as we know, the first such minimum principle that does not require any convexity hypothesis on the subequation in question.   For further background in this area the reader is referred to \cite{MP1}.\\

\noindent {\bf Organization: } Section \ref{sec:lipschitzargmin} is devoted to the proof of Theorem \ref{thm:intoargmin1}.    In \S\ref{sec:statement} we recall some standard terminology and notation concerning semiconvex functions, and use this to give a refined statement (Theorem \ref{thm:argmin1}) about calmness of the argmin function.   Theorem \ref{thm:intoargmin1} then follows immediately from this by Stepanov's Theorem (see Corollary \ref{cor:wisdiffae}).    In \S\ref{sec:propertiessemiconvex}  we collect some further properties of semiconvex functions, in preparation for \S\ref{sec:proofargminthm} in which we give a functional equation for the argmin function.  Then the proof of Theorem \ref{thm:argmin1} is given in \S\ref{sec:proofargmin1} using the Implicit Function Theorem for Lipschitz maps (which for completeness is proved in Appendix \ref{sec:appendiximplicit}).

In Section \ref{sec:Fsubharmonics} we summarize the basics of $F$-subharmonic functions in a way suited to our needs, including the idea of product subequations in \S\ref{sec:productsubequations}.    In Section \ref{sec:partialsup} we describe the partial sup-convolution, which is used in Section \ref{sec:fsubharmonicityofmarginals} to complete the proof of Theorem \ref{thm:introminimum}.

\section{Differentiability of the Argmin Function}\label{sec:lipschitzargmin}

\subsection{Statement}\label{sec:statement}

In this section we prove that the argmin function of a certain kind of semiconvex functions is differentiable (resp.\ calm) almost everywhere.    Suppose $\Omega\subset \mathbb R^{n+m}$ is open and let $\pi:\mathbb R^{n+m}\to \mathbb R^n$ be the projection, and write $$\Omega_x = \{ y\in \mathbb R^m : (x,y)\in \Omega\}.$$   We will assume throughout that $\Omega$ is convex and that each $\Omega_x$ is connected.  Now suppose
$$f:\Omega \to \mathbb R$$ and set
$$ g(x) : = \inf_{y\in \Omega_x} f(x,y) \text{ for } x\in \pi(\Omega).$$
\begin{definition}[Argmin]
The \emph{argmin function} is the set-valued function
$$ \argmin_f(x) : = \{ \gamma\in \Omega_x : \inf_{y\in \Omega_x} f(x,y) = f(x,\gamma)\}$$
where we allow the possibility that $\argmin_f(x)$ is empty.
\end{definition}

Below we shall make assumptions on $f$ that ensure that $\argmin(x)$ is everywhere defined and single-valued. In such cases we shall write
$$ \gamma(x) = \argmin_f(x)$$
so
$$ f(x,\gamma(x)) = \inf_{y\in \Omega_x} f(x,y) =g(x) \text{ for all } x\in \pi(\Omega).$$

The precise statement we will prove requires some terminology concerning subdifferentials.  Let $X\subset \mathbb R^n$ be open.

\begin{definition}
Suppose $g:X\to \mathbb R$.  For each $x_0\in  X$ define
$$\nabla_{x_0} g = \{ u\in \mathbb R^n: g(x) - g(x_0) \ge u.(x-x_0) \text{ for all }x \text{ sufficiently near } x_0 \}$$
which may be empty.  We call any $u\in \nabla_{x_0}g$ \emph{a lower support vector} for $g$ at $x_0$.  Similarly if $\kappa\in \mathbb R$ we let
$$ \nabla^\kappa_{x_0} g = \{u\in \mathbb R^n : g(x) - g(x_0) \ge u\cdot(x-x_0) - \frac{\kappa}{2} \|x-x_0\|^2 \text{ for all } x \text{ sufficiently near } x_0\}.$$
\end{definition}

\begin{definition}[Semiconvexity and Semiconcavity]
Let $\kappa\ge 0$.  We say $g:X\to \mathbb R$ is \emph{$\kappa$-semiconvex}   (resp. \emph{$\kappa$-semiconcave}) if $g(x) + \frac{\kappa}{2} \|x\|^2$ is convex (resp. $g(x) - \frac{\kappa}{2} \|x\|^2$ is concave).      If $g$ is $\kappa$-semiconvex/semiconcave for some $\kappa \ge 0$ then we say simply $g$ is \emph{semiconvex/semiconcave}.
\end{definition}

\begin{remark}
In the literature one will also find the term \emph{weakly-convex/concave} also used for semiconvex/semiconcave.
\end{remark}

 One can check that $g:X\to \mathbb R$ is locally $\kappa$-semiconvex if and only if $\nabla_{x_0}^{\kappa}g$ is non-empty for all $x_0$.  Moreover $g$ is differentiable at $x_0$ if and only if $\nabla_{x_0}g$ is a singleton, in which case its unique element is the derivative of $g$ at $x_0$.  Finally if $g$ is a convex function on a convex set $X$ then
$$ \nabla^\kappa_{x_0} g = \{u\in \mathbb R^n : g(x) - g(x_0) \ge u\cdot(x-x_0) - \frac{\kappa}{2} \|x-x_0\|^2 \text{ for all } x \}.$$

%\begin{definition} 
%Let $g:X\to \mathbb R$ and $\kappa\ge 0$.   For each %$x_0\in  X$ define
%$$ \nabla^\kappa_{x_0} g = \{u'\in \mathbb R^n : g(x) - g(x_0) \ge u'(x-x_0) - \frac{\kappa}{2} \|x-x_0\|^2\}$$
%\end{definition}
%So $g$ is locally $\kappa$-semiconvex if and only if $\nabla^\kappa_{x_0}g$ is non-empty for all $x_0$.

We now give a refined statement of Theorem \ref{thm:intoargmin1} that will be proved in section  \S\ref{sec:proofargminthm}.

\begin{theorem}[Argmin is calm almost everywhere]\label{thm:argmin1} \ 
Let $\Omega\subset \mathbb R^{n+m}$ be open, convex and so that $\Omega_x$ is connected for all $x$.  Also let $f:\Omega\to \mathbb R$ and suppose there are $\kappa\ge 0$ and $\sigma>0$ so that
%\begin{enumerate}[(i)]
%    \item 
%    There are $\kappa\ge 0$ and $\sigma>0$ so 
    
     \begin{equation}\label{eq:condagmin1}f(x,y) + \frac{\kappa}{2} \|x\|^2-\sigma \|y\|^2\text{  is convex and}\end{equation}
%\item 
 \begin{equation}\label{eq:condagmin2} 
\argmin_f(x) \text{ is non-empty for all } x\in \pi(\Omega).
  %\text{ exists and is single valued on }\pi(\Omega). 
  \end{equation}
%\end{enumerate}
Then 
\begin{enumerate}[(i)]
\item The function $$g(x):=\inf_y f(x,y) = f(x,\gamma(x)) \text{ for }x\in\pi(\Omega)$$ is $\kappa$-semiconvex and  $\gamma(x) := \argmin_f(x)$ is single valued for all $x\in \pi(\Omega)$.
\item  Given any $x_0\in \pi(\Omega)$ and $u_0\in \nabla_{x_0}^{\kappa} g$ there exists a Lipschitz function $$\phi:V\to \mathbb R$$ defined on a neighbourhood $V$ of $(x_0,u_0)$ in $\mathbb R^{2n}$ such that
$$ \gamma(x) = \phi(x,u) \text{ for all } (x,u)\in V \text{ with } u\in \nabla^\kappa_{x} g.$$
\item The function $\gamma$ is calm almost everywhere.  That is, for almost all $x_0\in \pi(\Omega)$ there are $C$ and $\delta>0$ such that
\begin{equation} \|\gamma(x) - \gamma(x_0)\| \le C \| x-x_0\|\text{ for } \|x-x_0\|<\delta.\label{eq:wlocallylipschitzae}\end{equation}
\end{enumerate}
\end{theorem}

\begin{corollary}[The Argmin is Differentiable Almost Everywhere]\label{cor:wisdiffae}
Under the hypothesis of the Theorem the  argmin function $\gamma$ is differentiable almost everywhere.
\end{corollary}
\begin{proof}
This follows from \eqref{eq:wlocallylipschitzae} and Stepanov's Theorem \cite{Stepanoff} (see also \cite[Theorem 3.4] {Heinonen_lectures}).
\end{proof}

The strategy of the proof of Theorem \ref{thm:argmin1} is to construct a functional equation satisfied by the argmin function, and then apply the implicit function theorem for Lipschitz functions.   In the next section we setup the necessary machinery to do so.

\subsection{Properties of semiconvex functions}\label{sec:propertiessemiconvex}

We collect a few basic statements about convex and semiconvex functions.    As above $\Omega\subset \mathbb R^{n+m}$ is open, convex and $\Omega_x$ is connected for all $x$.

\begin{lemma}\label{lem:semiconvexgrad}
Suppose $f:\Omega\to \mathbb R$ and $\tilde{f}(x,y) = f(x,y) + \kappa\frac{\|x\|^2}{2}$.   Then
$$ \nabla_{(x_0,y_0)} f = \nabla_{(x_0,y_0)} \tilde{f} + \kappa x_0$$
as sets. %In particular $f$ is $\kappa$-semiconvex if and only if $\nabla^\kappa_{(x_0,y_0)} f$ is non-empty for all $(x_0,y_0)\in \Omega$.
\end{lemma}
\begin{proof}
This is immediate from the definition, and left to the reader.
%The first statement follows directly from the definition.   The second is immediate as on a convex set being locally $%\kappa$-semiconvex is equivalent to be $\kappa$-semiconvex.
\end{proof}

\begin{lemma}\label{lem:basicconvexityI}
Suppose $f: \Omega\to \mathbb R$ and $h:\mathbb R^m\to \mathbb R$ and set $$\hat{f}(x,y) = f(x,y) + h(x) \text{ for } (x,y)\in \Omega.$$ Then 
\begin{equation}\argmin_{\hat{f}} (x) = \argmin_{f}(x) + h(x).\label{eq:argminshift}\end{equation}
In particular $\argmin_{\hat{f}}$ is single-valued if and only if $\argmin_f$ is single-valued.
\end{lemma}
\begin{proof}
Clearly
$$\argmin_{\hat{f}}(x) = \{ \gamma: \gamma = \inf_y \hat{f}(x,y) \} = \{ \gamma :\gamma= \inf_{y} f(x,y) + h(x)\} = \argmin_f(x) + h(x)$$
giving \eqref{eq:argminshift}.  The last statement follows immediately.  
\end{proof}

\begin{lemma}\label{lem:basicconvexityII}
Let $f:\Omega\to \mathbb R$  and suppose $f(x,y) + \frac{\kappa}{2}\|x\|^2$ is convex.   Then $g(x)  = \inf_y f(x,y)$ is $\kappa$-semiconvex.
\end{lemma}
\begin{proof}
Write $\tilde{f}(x,y) = f(x,y) + \frac{\kappa}{2} \|x\|^2$ so 
$$ g(x) +\frac{\kappa}{2} \|x\|^2 = \inf_y \tilde{f}(x,y)$$
which is the marginal function of convex function defined on  $\Omega$ (which is assumed to be convex and $\Omega_x$ is connected for each $x$).  Thus $g(x) + \frac{\kappa}{2} \|x\|^2$ is convex.
\end{proof}

\begin{lemma}[Gradient at argmin]\label{lem:gradientatargmax}
Suppose that $f:\Omega\to \mathbb R$ is convex and set $g(x) = \inf_y f(x,y)$.    Then  for all $x\in \pi(\Omega)$ and $\gamma\in \argmin_f(x)$
$$u\in \nabla_{x} g \Rightarrow (u,0)\in \nabla_{(x,\gamma)} f.$$
\end{lemma}
\begin{proof}
Suppose $\gamma\in \argmin_f(x)$ so $g(x) = f(x,\gamma)$.  Let $u\in \nabla_{x} g$.  Then for any $(x',y')\in \mathbb R^{n}\times \mathbb R^m$,
\begin{align}
f(x',y') - f(x,\gamma)& \ge g(x') - g(x) \\ &\ge u.(x'-x) = (u,0). ((x',y') - (x,\gamma))
\end{align}
so $(u,0)\in \nabla_{(x,\gamma)}f$ as claimed.
\end{proof}

The next statement is a slight modification of \cite[Proposition 6.4]{Howard}.  
\begin{proposition}\label{prop:thefunctionsGH}
Let $\sigma>0$ and suppose $f:\mathbb R^{n+m}\to \mathbb R$ is such that $f(x,y) - \frac{\sigma}{2} \|y\|^2$ is convex.    Define the set-valued function
$$ G(p) = p + \nabla_{p} f \text{ for } p=(x,y)\in \mathbb R^{n+m}.$$
Then 
\begin{enumerate}[(i)]
\item $G$ is non-contractive.  That is, if $\zeta_i \in G(p_i)$ for $i=1,2$ then 
\begin{equation}\label{eq:Fnoncontract}\| \zeta_1-\zeta_2\| \ge \|p_1-p_2\|.\end{equation}
\item There exist a single-valued function $H:\mathbb R^{n+m} \to \mathbb R^{n+m}$ that is inverse to $G$, by which we mean
\begin{equation} H(\zeta) = p \Longleftrightarrow \zeta \in G(p).\label{eq:inverse}\end{equation}
\item The function $H$ is Lipschitz with Lipschitz constant $1$.  Moreover there is a $\mu<1$ such that letting $\pi_2:\mathbb R^{n+m}\to \mathbb R^m$ denote the second projection, 
\begin{equation}
\| \pi_2 H(\zeta_1)-\pi_2 H(\zeta_2)\| \le \mu \| \zeta_1-\zeta_2\| \text{ for all } \zeta_1,\zeta_2\in\mathbb R^{n+m}.\label{eq:inequalitylambda}
\end{equation}
\end{enumerate}
\end{proposition}

\begin{proof}
Let $p_i:= (x_i,y_i)\in \mathbb R^{n+m}$ for $i=1,2$.  We first claim
\begin{equation}\label{eq:thefunctionsGH1}
 (\nabla_{p_2} f- \nabla_{p_1}f ).(p_2-p_1) \ge \sigma \|y_2-y_1\|^2 \text{ for all } (x_i,y_i)\in \mathbb R^{n+m}.
 \end{equation}
 To see this, let $\tilde{f}(x,y) = f(x,y) - \frac{\sigma}{2} \|y\|^2$ which by assumption is convex and $\nabla_{(x,y)} \tilde{f} = \nabla_{(x,y)} f - (0,\sigma y)$.   Then
 \begin{equation}\label{eq:thefunctionsGH2}
 \tilde{f}(p_1) - \tilde{f}(p_2) \ge \nabla_{p_2}\tilde{f}.(p_1-p_2) = \nabla_{p_2} f.(p_1-p_2) - \sigma y_2.(y_1-y_2).
 \end{equation}
Swapping the indices we also have
 \begin{equation}\label{eq:thefunctionsGH3}
 \tilde{f}(p_2) - \tilde{f}(p_1) \ge \nabla_{p_1}\tilde{f}.(p_2-p_1) = \nabla_{p_1} f.(p_2-p_1) - \sigma y_1.(y_2-y_1).
 \end{equation}
 Adding \eqref{eq:thefunctionsGH2} and \eqref{eq:thefunctionsGH3} and rearranging gives \eqref{eq:thefunctionsGH1}.
 
 Now from Cauchy-Schwarz and \eqref{eq:thefunctionsGH1}
 \begin{align}
   \| G(p_1) - G(p_2)\| \|p_1-p_2\| &\ge  ( G(p_1) - G(p_2)).(p_1-p_2) \nonumber \\
 &= (p_1 -p_2 + \nabla_{p_1} f- \nabla_{p_2}f).(p_1-p_2) \nonumber\\
 &\ge \|p_1-p_2\|^2 + \sigma \|y_1-y_2\|^2\label{eq:strict}\\
 &\ge \|p_1-p_2\|^2 \nonumber\end{align}
which in particular implies (i).

We claim next that $G$ is surjective, by which we mean for all $\zeta\in \mathbb R^{n+m}$ there is an $p\in \mathbb R^{n+m}$ such that $\zeta\in G(p)$.    To see this let  
$$\phi(p):= \frac{1}{2} \|p\|^2 + f(p) -p.\zeta.$$
The function $p\mapsto \frac{1}{2}\|p\|^2 -p.\zeta$ is convex, and hence so is $\phi$ and
$$\nabla_{p_0} \phi = \nabla_{p_0} f+ p_0 - \zeta = G(p_0)-\zeta.$$
Similarly the function
$$\psi(p): = \frac{1}{4} \|p\|^2 + f(p)  -p.\zeta = \phi(p) - \frac{1}{4}\|p\|^2$$ is convex.  Pick $b\in \nabla_{0} \psi$ so $\psi(p) -\psi(0) \ge b.p$ giving
$$\phi(p) \ge \phi(0) + \frac{1}{4}\|p\|^2.$$
As $\phi$ is continuous this implies $\phi$ has a global minimum at  some $p_0\in \mathbb R^{n+m}$, and so $0$ is a lower support vector for $\phi$ at $0$.   Thus $0\in \nabla_{p_0}\phi = G(p_0)-\zeta$ implying that $\zeta\in G(p_0)$.  Thus $G$ is surjective as claimed.

In particular the inverse $H$ to $G$ defined by 
$$ H(\zeta) = \{ p\in \mathbb R^{n+m}: \zeta\in G(p)\}$$ 
is non-empty, and $G$ being non-contractive implies that it is single-valued.  That $H$ has Lipshitz constant $1$ follows from (i).  

Finally given $\zeta_1,\zeta_2$ set $p_i:=(x_i,y_i):= H(\zeta_i)$ so by definition $\zeta_i\in G(p_i)$ and $y_i=\pi_2 H(\zeta_i)$ .  %From \eqref{eq:strict}
%\begin{equation}\|\zeta_1-\zeta_2\| \ge \|p_1-p_2\| + \sigma \frac{\|y_1-y_2\|^2}{\|p_1-p_2\|}.
%\end{equation}
To ease notation let $\alpha: = \|x_1-x_2\|$ and $\beta: = \|y_1-y_2\|  = \| \pi_2 H(\zeta_1) - \pi_2 H(\zeta_2)\|$.  Then dividing \eqref{eq:strict} by $\|p_1-p_2\|$ gives
 $$ \|\zeta_1-\zeta_2\|  \ge (\alpha^2 + \beta^2)^{1/2}+ \sigma \frac{\beta^2}{(\alpha^2 + \beta^2)^{1/2}}.$$
If $\alpha \ge \sigma \beta$ then $\|\zeta_1-\zeta_2\|\ge (1+\sigma^2)^{1/2} \beta$.  If $\alpha\le \sigma \beta$ then
\begin{align*} \|\zeta_1-\zeta_2\|  &\ge \beta + \sigma \frac{\beta^2}{(\sigma^2\beta^2+ \beta^2)^{1/2}}\\
&= (1 + \frac{\sigma}{(1+\sigma^2)^{1/2}}) \beta.
\end{align*}
Hence \eqref{eq:inequalitylambda} holds with $\mu: = \min\{ (1+\sigma^2)^{1/2}, (1 + \frac{\sigma}{(1+\sigma^2)^{1/2}})\}^{-1}<1$. 
\end{proof}

We will also need the following simpler corollary (which is proved in the same way, or follows formally from Proposition \ref{prop:thefunctionsGH} upon taking $m=0$).

\begin{corollary}\label{cor:thefunctionsGH}
Suppose $g:\mathbb R^n\to \mathbb R$ is convex and define the set-valued function
$$ G_1(x) = x + \nabla_{x} g \text{ for } x\in \mathbb R^n.$$
Then 
\begin{enumerate}
\item $G_1$ is non-contractive, that is 
 \begin{equation}\|G_1(x_1) - G_1(x_2)\| \ge \|x_1 -x_2\| \text{ for all } x_1,x_2\in X.\label{eq:hefunctionsGH:Gnoncontractive:repeat}\end{equation}
%That is, if $\zeta_i \in G_1({x_i})$ for $i=1,2$ then
%\begin{equation}\label{eq:Fnoncontract}\| \zeta_1-\zeta_2\| \ge (1+\sigma)\|x_1-x_2\|.\end{equation}
\item There exist a single-valued function $H_1:\mathbb R^{n} \to \mathbb R^{n}$ that is inverse to $G_1$, and $H_1$ is Lipschitz with Lipschitz constant 1.
%by which we mean
%\begin{equation} H_1(\zeta) = x \Leftrightarrow \zeta \in G(x).\label{eq:inverse}\end{equation}
%Moreover $H_1$ is Lipschitz with Lipschitz constant $1$.
\end{enumerate}
\end{corollary}

\subsection{Functional Equation for argmin}\label{sec:proofargminthm}
Suppose now that $f:\mathbb R^{n+m} \to \mathbb R$ is convex and as usual let $g(x) = \inf_y f(x,y)$ which is also convex.    Consider the set-valued functions
\begin{align*} G_1(x) &= x + \nabla_x g,\\
 G(x,y) &= (x,y) + \nabla_{(x,y)} f.\end{align*}
By Proposition \ref{prop:thefunctionsGH} and Corollary \ref{cor:thefunctionsGH} these have single-valued inverses $H_1:\mathbb R^{n}\to \mathbb R^{n}$ and $H:\mathbb R^{n+m}\to \mathbb R^{n+m}$.    That is
\begin{align} H_1(u) = x & \Leftrightarrow u\in G_1(x) \text{ for } x,u\in \mathbb R^{n}\label{eq:Hinverse1}\\
H(u,v) = (x,y) &\Leftrightarrow (u,v)\in G(x,y) \text{ for } (x,y), (u,v) \in \mathbb R^{n+m}.\label{eq:Hinverse2}
\end{align}
We use these to define a functional equation for $\argmin_f$.  Let
$$J:\mathbb R^{n}\times \mathbb R^{n}\times \mathbb R^{m} \to \mathbb R^{m}$$
\begin{equation}\label{def:J}J(x,u,y) := y - \pi_2 H(H_1(x+u)+u,y).\end{equation}

\begin{proposition}[Functional Equation for argmin]\label{prop:Jargmin}
Suppose that $f(x,y)$ is convex and let $g(x) = \inf_y f(x,y)$.  
Then
$$ J(x,\nabla_x g,\argmin_f(x)) =0 \text{ for all } x\in \mathbb R^n.$$
That is,
$$ J(x,u,\gamma) =0 \text{ for all } x\in \mathbb R^n \text{ and  }\gamma\in \argmin_f(x) \text{ and }u\in \nabla_x g.$$
\end{proposition}
\begin{proof} 
Let $x\in\mathbb R^{n}$, $\gamma\in \argmin_f(x)$ and $u\in \nabla_x g$.   Then $x+u\in G_1(x)$ so \eqref{eq:Hinverse1} gives $H_1(x+u)=x$.  On the other hand since $\gamma\in \argmin_f(x)$ we have by Lemma  \ref{lem:gradientatargmax}, $$(u,0) \in \nabla_{(x,\gamma)} f.$$
 Thus $$ (x,\gamma) + (u,0)= (u+x,\gamma)\in G(x,\gamma)$$  so \eqref{eq:Hinverse1}  gives $H(u+x,\gamma) = (x,\gamma)$.  So
%\begin{align*} J(x,u,\gamma) &= (H_1(x+u),\gamma) - H(H_1(x+u)+u,\gamma)\\
%&= (x,\gamma)  - H(x+u,\gamma) = (x,\gamma) - (x,\gamma) \\&= 0\end{align*}
\begin{align*} J(x,u,\gamma) &= \gamma - \pi_2 H(H_1(x+u)+u,\gamma)\\
&= \gamma  - \pi_2 H(x+u,\gamma) = \gamma - \pi_2(x,\gamma) \\&= 0\end{align*}
as claimed.
\end{proof}

We next collect two basic properties of $J$:

\begin{lemma}[Properties of $J$]\label{lem:propertiesJ}
The function $J$ is Lipschitz in $(x,u,y)$.    Moreover if $f(x,y) - \frac{\sigma}{2}\|y\|^2$ is convex for some $\sigma>0$ then there is a $\lambda>0$ such that for fixed $x,u$
$$ \|J(x,u,y_1) - J(x,u,y_2)\| \ge   \lambda \|y_1-y_2\| \text{ for all } y_1,y_2.$$
\end{lemma}
\begin{proof}
Clearly $J$ is Lipschitz in all variables since both $H$ and $H_1$ are.  For the second statement, suppose $f(x,y) - \frac{\sigma}{2}\|y\|^2$ is convex and let $\pi_2:\mathbb R^{n}\times \mathbb R^{m}\to \mathbb R^{m}$ be the second projection.  We know from Proposition \ref{prop:thefunctionsGH}(iii) that there is a $\mu<1$ such that
\begin{align}
 \| \pi_2H(v,y_1) - \pi_2H(v,y_2)\| &\le \mu  \|y_1-y_2\| \text{for all } v,y_1.\end{align}
Now fix $x,u$ and let $v:=H_1(x+u)+u$.  Then if $y_1,y_2\in \mathbb R^{m}$,
%$$ J(x,u,y_2) - J(x,u,y_1) = y_2-y_2 - \pi_2H(v,y_2) + \pi_2H(v,y_1)$$
%and so
\begin{align*}
 \|J(x,u,y_2) - J(x,u,y_1) \|  &= \| y_2-y_1 - \pi_2H(v,y_2) + \pi_2 H(v,y_1)\|\\
 &\ge \| y_2 -y_1\| - \| \pi_2H(v,y_2) - \pi_2 H(v,y_1)\|\\
 %&\ge \| y_2 -y_1\| - \frac{1}{1+\sigma} \|y_1-y_2\|\\
&\ge (1-\mu) \|y_2-y_1\|.
\end{align*}

\end{proof}

\subsection{Statement of Alexandrov's Theorem}

Let $X\subset \mathbb R^n$ be open. The following is a precise version of Alexandrov's Theorem:

%\begin{definition}[Differentiability of set-valued functions]
%We say a set-valued function $f:X\to   2^\mathbb R$ is \emph{differentiable} at $x_0\in X$ if there exists a linear $L:\mathbb R^n\to \mathbb R^n$ such that the following holds: for all $\epsilon>0$ there is a $\delta>0$ such that
%\begin{equation} \| u-u_0 - L(x-x_0)|\| \le \epsilon \|x-x_0\| \text{ for all }u\in F(x), u_0\in F(x_0)\text{ and }\|x-x_0\|<\delta.
%\label{eq:twicediff:again}\end{equation}
%\end{definition}

\begin{theorem}[Alexandrov's Theorem]\label{thm:alexandrov:repeat}
Let $g:X\to \mathbb R$ be locally convex.   Then the set-valued function $$x\mapsto \nabla_{x} g$$ is differentiable at $x_0$ for almost all $x_0$ in $X$.   That is, for almost all $x_0$ there is an $L\in \Hom(\mathbb R^{n},\mathbb R^{m})$ such that for all $\epsilon>0$ there is a $\delta>0$ such that for $\|x-x_0\|<\delta$  we have
\begin{equation}\label{eq:twicediffgradient} \| u-u_0 - \frac{L}{2}(x-x_0)|\| \le \epsilon \|x-x_0\| \text{ for all }u\in \nabla_{x}g \text{ and } u_0\in \nabla_{x_0}g.
\end{equation}
Moreover for almost all $x_0$ the function $g$ is twice differentiable at $x_0$ and $\Hess_{x_0}(g)=L$.  That is, for any $\epsilon>0$ there is a $\delta>0$ such that 
\begin{equation}\label{eq:twicediffhessian} |g(x) - g(x_0) - \nabla g|_{x_0}.(x-x_0) - \frac{1}{2}(x-x_0)^t \Hess_{x_0}(g) (x-x_0)\rangle | \le \epsilon \|x-x_0\|^2\end{equation}
for all $\|x-x_0\|<\delta.$
\end{theorem}
\begin{proof}

This originates in \cite{Alexandrov} and for an exposition the reader is referred to \cite[Theorems 6.1,7.1]{Howard}.  (We remark that the latter cited work requires the function to be convex and defined on all of $\mathbb R^n$;  but the statement we want is local, and being locally convex, $g$ is also locally Lipschitz \cite{WayneState}, and so using \cite[Theorem 4.1]{MinYan} we know that $X$ is covered by small open sets $U$ such that $g|_U$ extends to a convex function on $\mathbb R^n$  so the cited work applies.)
%Every Convex Function is Locally Lipschitz Wayne State University, Mathematics Department Coffee Room
\end{proof}

\subsection{Proof of Theorem \ref{thm:argmin1}}\label{sec:proofargmin1}

%We continue to have $\Omega\subset \mathbb R^{n+1}$

\begin{lemma}[Continuity of argmin]\label{lem:continuityargmin}
Let $\Omega\subset X\times \mathbb R$ be convex and such that $\Omega_x$ is connected for each $x\in X$.    Let $f:\Omega\to \mathbb R$ be continuous, and suppose that for each $x\in X$ the function $y\mapsto f(x,y)$ is strongly convex and attains its minimum at some point.   Then $\gamma(x) = \argmin_f(x)$ is single valued and continuous.
\end{lemma}
\begin{proof}
For fixed $x$ the hypothesis imply that $y\mapsto f(x,y)$ is a strongly convex function on the connected set $\Omega_x$ that attains its minimum, and thus this minimum $\gamma(x)$ must be a unique.  We first claim that $\gamma$ is locally bounded.  Fix $x_0\in X$ and let $a:= \gamma(x_0)$.  Then by strong convexity there is an $\epsilon>0$ and $c>0$ such that $f(x_0,y)>a+\epsilon$ if $\|y-\gamma(x_0)\|\ge c$.   By continuity we may take $\delta>0$ small so if $\|x-x_0\|<\delta$ and $\|y-\gamma(x_0)\|=c$ then $f(x,y)>a+\epsilon$ and, and furthermore that $f(x,\gamma(x_0))<a+\epsilon$.  But by strict convexity of $y\mapsto f(x,y)$ this implies $\gamma(x) \in [\gamma(x_0)-c,\gamma(x_0)+c]$ for all $\|x-x_0\|<\delta$, and thus $\gamma$ is locally bounded.

Now suppose $(x_n)$ is a sequence in $X$ converging to $x$ as $n\to \infty$.   By the above we may assume $S:=\{\gamma(x_n)\}$ is bounded.   Let $b$ be a cluster point of $S$, so there is a subsequence $x_{n_r}$ with $\gamma(x_{n_r})\to b$ as $r\to \infty$.  By continuity of $f$ for any $y\in \mathbb R^m$,
$$ f(x,b)= \lim_{r\to \infty} f(x_{n_r},\gamma(x_{n_r})) \le \lim_{r\to \infty} f(x_{n_r}, y) = f(x,y).$$
Hence $b=\gamma(x)$.  As this holds for all cluster points of $S$ we deduce $\gamma(x_n)\to \gamma(x)$ as $n\to \infty$, proving continuity of $\gamma$.
\end{proof}

\begin{proof}[Proof of Theorem \ref{thm:argmin1}]
We first claim that there is no loss in generality in assuming that $\Omega= \mathbb R^{n+m}$.    To see this, suppose $f:\Omega\to \mathbb R$ has properties \eqref{eq:condagmin1} and \eqref{eq:condagmin2}.   Then $\gamma=\argmin_f$ is single-valued and continuous (Lemma \ref{lem:continuityargmin}).   So given $x_0\in \pi(\Omega)$ there are small balls $x_0\in U\subset \pi(\Omega)$ and $\gamma(x_0)\in V\subset \mathbb R^m$ so that $U\times V\subset \Omega$ and $\gamma(U)\subset V$.  Moreover as $f$ is semiconvex, by shrinking $U,V$ we may assume that $f|_{U\times V}$ is Lipschitz (all convex functions are Lipschitz, see e.g.\ \cite{WayneState}).    Let $\tilde{f}(x,y) : = f(x,y) + \frac{\kappa}{2} \|x\|^2 - \frac{\sigma}{2} \|y\|^2$ which we are assuming is convex on $\Omega$.  Then \cite[Theorem 4.1]{MinYan} we know $\tilde{f}|_{U\times V}$ extends to a convex function $\tilde{h}$ on all of $\mathbb R^{n+m}$.  Now let $$h(x,y): = \tilde{h}(x,y) - \frac{\kappa}{2}\|x\|^2 + \frac{\sigma}{2} \|y\|^2.$$
For fixed $x$ the convex function $y\mapsto h(x,y)$ agrees with the function $y\mapsto f(x,y)$ when $y\in V$.  Since $V$ contains $\gamma(x)=\argmin_f(x)$, this implies $\argmin_h(x) = \argmin_f(x)=\gamma(x)$.   Hence $h$ satisfies the hypothesis of the Theorem with $\Omega = \mathbb R^{n+m}$, and so $\gamma|_U$ has the properties in the conclusion of the theorem (which are all local), which proves the claim.

So from now on assume $f:\mathbb R^{n+m} \to \mathbb R$ satisfies \eqref{eq:condagmin1} and \eqref{eq:condagmin2}.    Consider first the case $\kappa=0$, so $(x,y)\mapsto f(x,y) -\frac{\sigma}{2} \|y\|^2$ is convex.  Then in particular $f$ is convex, and so $g(x) = \inf_y f(x,y)$ is also convex.  Moreover for fixed $x$ the function $y\mapsto f(x,y)$ is strictly convex, and so $\argmin_f$ (which is assumed to be non-empty) must be single valued.  Consider the functional $J$ from \eqref{def:J} so by Proposition \ref{prop:Jargmin}
\begin{equation}J(x,u,\gamma(x)) =0  \text{ for all } x \text{ and } u\in \nabla_x g.\label{eq:functionalJrepeat}\end{equation}
Fix $x_0\in \mathbb R^n$ and $u_0\in \nabla_{x_0}g$ so $J(x_0,u_0,\gamma(x_0)) =0$.  The properties of $J$ proved in Lemma \ref{lem:propertiesJ} mean we can apply the Inverse-function Theorem for Lipschitz maps (for convenience of the reader we give a proof of this in Appendix \ref{sec:appendiximplicit}, and apply it here with $r$ replaced with $2n$ and $s$ replaced with $m$).    This yields a Lipschitz function $\phi:V\to \mathbb R^{m}$ defined on a neighbourhood $V$ of $(x_0,u_0)$ such that
$$ J(x,u,y)=0 \Leftrightarrow y = \phi(x,u).$$
This combined with \eqref{eq:functionalJrepeat} gives
$$\gamma(x) = \phi(x,u) \text{ for all } (x,u)\in V \text{ with } u\in \nabla_x g.$$

We next prove $\gamma$ is calm almost everywhere.   As $g$ is convex we have by Alexandrov's Theorem \eqref{eq:twicediffgradient} that for almost all $x_0$ there are $\delta_1>0$ and linear $L:\mathbb R^{n}\to \mathbb R^{n}$ such that for $\|x-x_0\|<\delta_1$ 
\begin{equation}\| u  - u_{0} \| \le (1 + \|L\|) \|x-x_0\| \text{ for all } u\in \nabla_{x}g \text{ and } u_0\in \nabla_{x_0} g.\label{eq:argmax2eq1}
\end{equation}

Pick $u_0\in \nabla_{x_0} g$, and let $\phi:V\to \mathbb R$ be the Lipschitz function constructed above.  For concreteness say that $V$ contains the set $\|x-x_0\|<\delta_2$ and $\|u-u_0\|<\delta_2$ and that $\phi$ has Lipschitz constant $C'$ there.  Thus
$$ \gamma(x) = \phi(x,u) \text{ for }  \|x-x_0\|<\delta_2, \|u-u_0\|<\delta_2 \text{ and } u\in \nabla_x g.$$  
Set $$\delta :=\min\{\delta_1, \frac{\delta_2}{1+\|L\|}\}$$ and suppose $\|x-x_0\|<\delta$.  Picking any $u\in \nabla_x g$, by \eqref{eq:argmax2eq1} $\|u-u_0\|<\delta_2$ and so
$$ \|\gamma(x)  -\gamma(x_0)\|  = \|\phi(x,u) - \phi(x_0,u_0)\| \le C' (\|x-x_0\| + \|u-u_0\|) \le C'(2 + \|L\|) \|x-x_0\|.$$
Thus $\gamma$ is calm $x_0$.

The case of general $\kappa$ is easily reduced to the case $\kappa=0$.  For suppose $f(x,y) + \frac{\kappa}{2} \|x\|^2  -\frac{\sigma}{2}|y|^2$ is convex and $\argmin_f$ is single-valued.     Set
$$\tilde{f}(x,y) = f(x,y) + \frac{\kappa}{2} \|x\|^2$$
Then  $\tilde{f}(x,y) - \frac{\sigma}{2} \|y\|^2$ is convex, and by \eqref{eq:argminshift}
$$\argmin_{\tilde{f}}(x) = \argmin_f (x).$$
Thus $\argmin_{\tilde{f}}$ is also single-valued, so by the above the Theorem can be applied to $\tilde{f}$.  Let
$$ \gamma(x) := \argmin_{\tilde{f}}(x) = \argmin_f (x)$$ 
Setting $\tilde{g}(x):=\inf_y \tilde{f}(x,y)$, given $x_0$ and $u_0\in \nabla_{x_0} \tilde{g}$ we know that there is a locally Lipschitz function $\tilde{\phi}:\tilde{V}\to \mathbb R$ defined on a neighbourhood $\tilde{V}$ of $(x_0,u_0)$ such that
$$ \gamma(x) = \tilde{\phi}(x,u) \text{ for } (x,u)\in \tilde{V} \text{ with } u\in \nabla_x \tilde{g}.$$
Set $\phi(x,u) = \tilde{\phi}(x,u+\kappa x)$ which is locally Lipschitz around $(x_0,u_0 + \kappa x_0)$.  And if $u\in \nabla^{\kappa}g$ then $u-\kappa x_0 \in \nabla_x\tilde{g}$ so $\gamma(x) = \tilde{\phi}(x,u-\kappa x_0)  = \phi(x,u)$.      Thus the conclusion of the Theorem also holds for $f$ and we are done.
\end{proof}

\section{F-subharmonic functions}\label{sec:Fsubharmonics}
\subsection{Basic definitions}\label{sec:basics} We summarise some basic properties of F-subharmonic functions from the work of Harvey-Lawson.  We refer the reader to \cite{MP1} for a more detailed summary, or the original papers \cite{HL_Dirichletduality,HL_Dirichletdualitymanifolds}.  Let $X\subset \mathbb R^{n}$ be open and 
$$J^2(X):=X\times \mathbb R\times \mathbb R^n \times \Sym^2_{n} = X\times J^2_n$$
be the jet-bundle over $X$.  For $F\subset J^2(X)$ and $x\in X$ we write 
$$ F_x= \{ (r,p,A)\in J^2_n  : (x,r,p,A)\in F\}.$$

\begin{definition}[Primitive Subequations]
We say that $F\subset J^2(X)$ is a \emph{primitive subequation} if 
\begin{enumerate}
\item (Closedness) $F$ is closed.
\item (Positivity) 
\begin{equation}\label{eq:positivity} 
(r,p,A)\in F_x \text{ and } P\in \Pos_{n} \Rightarrow (r,p,A+P)\in F_x.\end{equation}
\end{enumerate}
We say that $  F\subset J^2(X)$ has the \emph{Negativity Property} if
\begin{enumerate}\setcounter{enumi}{2}
\item (Negativity) \begin{equation}\label{eq:negativity} 
(r,p,A)\in F_x \text{ and } r'\le r \Rightarrow (r',p,A)\in F_x.\end{equation}
%\item (Topological) \begin{equation}\label{eq:topological} F = \overline{\Int( F)} \text{ and } F_x = \overline{\Int F_x}.\text{ and } \Int F_x = (\Int F)_x\text{ for all }x\in X\end{equation}
\end{enumerate}
%where the bar denotes topological closure.
%We say that $ F$ is \emph{convex} if each $F_x$ is convex,  i.e. if $\alpha_1,\alpha_2\in  F_x$ and $t\in [0,1]$ then $t\alpha_1 + (1-t)\alpha_2\in F_x$.   %We say it is a \emph{cone} if each $F_x$ is cone over the origin,  i.e. if $\alpha\in  F_x$ then $t\alpha\in F_x$ for all $t\ge 0$.
\end{definition}

%\begin{remark}
%The majority of the results of this paper hold for primitive subequations that satisfy the Negativity Property.    The extra assumption of being a subequation is important for the Comparison Principle proved in %\cite{HL_Dirichletduality,HL_Dirichletdualitymanifolds}.
%\end{remark}

\begin{definition}[Upper contact points, Upper contact jets]
Let $$f: X\to \mathbb R\cup\{-\infty\}.$$ We say that $x\in X$ is an \emph{upper contact point} of $f$ if $f(x)\neq -\infty$ and there exists $(p,A)\in \mathbb R^n\times \Sym^2_n$ such that
$$f(y)\le f(x) + p.(y-x) + \frac{1}{2} (y-x)^tA(y-x) \text{ for all } y \text{ sufficiently near } x.$$
When this holds we refer to both $(f(x),p,A)$ and $(p,A)$ as an \emph{upper contact jet} of $f$ at $x$.
\end{definition}

\begin{definition}[F-subharmonic function]
Suppose $F\subset J^2(X)$.   We say that an upper-semicontinuous function $f:X\to \mathbb R\cup \{-\infty\}$ is \emph{F-subharmonic} if $$ (f(x),p,A)\in F_x \text{ for all upper contact jets }  (p,A) \text{ of }f \text{ at } x.$$
%and each upper contact jet $(p,A)$ of $f$ at $x$ it holds that $(f(x),p,A)\in F_x$.  
We let $F(X)$ denote the set of $F$-subharmonic functions on $X$.
\end{definition}

Clearly being $F$-subharmonic is a local condition on $X$.
%by which we mean that if $\{X_{\alpha}\}_{\alpha\in \mathcal A}$ is an open cover of $X$ then $f\in F(X)$ if and only if $f\in F(X_{\alpha})$ for all $\alpha$.  

\begin{proposition}\label{prop:basicproperties}
Let $F\subset J^2(X)$ be closed.  Then 
\begin{enumerate}
\item (Maximum Property) If $f,g\in F(X)$ then $\max\{f,g\}\in F(X)$.
\item (Decreasing Sequences) If $f_j$ is decreasing sequence of functions in $F(X)$ (so $f_{j+1}\le f_j$ over $X$) then $f:=\lim_j f_j$ is in $F(X)$.
\item (Uniform limits) If $f_j$ is a sequence of functions on $F(X)$ that converge locally uniformly to $f$ then $f\in F(X)$.
\item (Families locally bounded above) Suppose $\mathcal F\subset F(X)$ is a family of $F$-subharmonic functions locally uniformally bounded from above.  Then the upper-semicontinuous regularisation of the supremum
$$ f:= {\sup}^*_{f\in \mathcal F} f$$
is in $F(X)$.
\item If $F$ is constant coefficient and $f$ is $F$-subharmonic on $X$ and $x_0\in \mathbb R^{n}$ is fixed,  then the function $x\mapsto f(x-x_0)$ is $F$-subharmonic on $X-x_0$.
\end{enumerate}
\end{proposition}
\begin{proof}
See \cite[Theorem 2.6]{HL_Dirichletdualitymanifolds} for (1-4).  Item (5) is immediate.
\end{proof}

\begin{definition}
Let $F\subset J^2(X)$.  
\begin{enumerate}
\item We say $F$ is \emph{constant coefficient} if $F_x$ is independent of $x$, i.e.\
$$ (x,r,p,A) \in F_x \Leftrightarrow (x',r,p,A)\in F_{x'} \text{ for all }x,x',r,p,A.$$
%\item We say $F$ is \emph{independent of the gradient part} (or \emph{gradient-independent})  if each $F_x$ is independent of $p$, i.e. 
%$$ (r,p,A)\in F_x \Leftrightarrow (r,p',A) \in F_x \text{ for all }  x,r,p,p',A.$$
\item We say $F$ \emph{depends only on the Hessian part} if each $F_x$ is independent of $(r,p)$,  i.e.
$$ (r,p,A) \in F_x \Leftrightarrow (r',p',A)\in F_{x} \text{ for all }x,r,r',p,p',A.$$
\end{enumerate}
\end{definition}

An important example is $$\calPos:= \{ (x,r,p,A)\in J^2(X): A \text{ is semipositive}\}$$ which is a constant-coefficient primitive subequation that depends only on the Hessian part.  Then $\calPos$-subharmonic functions are precisely those that are locally convex  \cite[Example 14.2]{HL_Dirichletdualitymanifolds}.

\begin{lemma}[Sums of $F$-subharmonic and convex functions]\label{lem:sumsubharmoniconvex}
Suppose $F\subset J^2(X)$ is a constant coefficient primitive subequation that depends only on the Hessian part.  If $f$ is $F$-subharmonic on $X$ and $g$ is a convex quadratic function on $X$, then $f+g$ is $F$-subharmonic.
\end{lemma}
\begin{proof}
The hypothesis is that $g(x) = a + b.x + \frac{1}{2} x^tC x$ for some $a,b\in \mathbb R^n$ and some semipositive symmetric matrix $C$.  One can check that if $(p,A)$ is an upper-contact point of $f+g$ at $x$ then $(x^tC +p-b, A-C)$ is an upper-contact jet for $f$ at $x$.  As $f$ is $F$-subharmonic this implies $(f(x), x^tC +p-b, A-C)\in F$.  Since $F$ depends only on the Hessian part, and satisfies the Positivity property, this in turn implies $(f(x)+g(x),p,A)\in F$ proving that $f+g$ is $F$-subharmonic as required.
\end{proof}

\subsection{Product Subequations}\label{sec:productsubequations}

For $\Gamma\in \Hom(\mathbb R^{n},\mathbb R^{m})=M_{m\times n}(\mathbb R)$ consider
\begin{align}i_\Gamma:\mathbb R^{n} &\to \mathbb R^{n+m} \quad i_\Gamma(x) = (x,\Gamma x)\\
j : \mathbb R^{m} &\to \mathbb R^{n + m} \quad j(y) = (0,y).\end{align}
which induce natural pullback maps
\begin{equation}
i_\Gamma^*: J^2_{n+m}\to J^2_{n} \text{ and }
j^* : J^2_{n+m}\to J^2_{m}.\end{equation}
%\begin{align*}
%i_U^* : \mathbb R\times \mathbb R^{n+m} \times \Sym_{n+m}^2 &\to \mathbb R\times \mathbb R^{n} \times \Sym^2_{n}.\\
%j^* : \mathbb R\times \mathbb R^{n+m} \times \Sym_{n+m}^2 &\to \mathbb R\times \mathbb R^{m} \times \Sym^2_{m}.
%\end{align*}
%with  the following property.  If $f:\mathbb R^{n+m}\to \mathbb R$ is twice differentiable at $(0,0)\in \mathbb R^{n+m}$ and we let \begin{align*}f_1(x) &:= f(x,\Gamma x) \text{ for } x\in \mathbb R^n\\f_2(y) &: = f(0,y)\text{ for } x\in \mathbb R^m\end{align*}
%then $f_1,f_2$ are twice differentiable at $0\in \mathbb R^{n}$ and $0\in \mathbb R^{m}$ respectively, and
%\begin{align}
% J^2_0 (f_1) &= i_\Gamma^* J^2_{(0,0)} (f)\\\
%  J^2_{0} (f_2) &= j^* J^2_{(0,0)} (f).
% \end{align}
We can write these explicitly.  Suppose
$$ p := \left(\begin{array}{c} p_1 \\ p_2 \end{array} \right)\in \mathbb R^{n+ m} \text{ and } A := \left(\begin{array}{cc} B & C \\ C^t & D \end{array}\right)\in \Sym^2_{n+m}$$
where the latter is in block form, so $B\in \Sym^2_{n}$ and $D\in \Sym^2_{m}$.  Then
\begin{align}\label{eq:defofistar}
i_\Gamma^*(r,p,A) &= \left(r, p_1 + \Gamma^t p_2, B + C\Gamma + \Gamma^t C^t + \Gamma^tD\Gamma\right)\\
j^*(r,p,A)  &= (r,p_2,D).
\end{align}

\begin{comment}
\begin{definition}[Slices]
Given ${x_0}\in  X$ we call
$$ i_{x_0} : \{ y\in Y : (x_0,y)\in X\times Y \} \to X\times Y  \quad i_{x_0}(y) = (x,y)$$
a \emph{vertical slice} of $X\times Y$.  
Given $y_0\in Y$ and $\Gamma\in Hom(\mathbb R^{n},\mathbb R^{m})$ we call
$$ i_{y_0,\Gamma} : \{ x : (x,y_0 + \Gammax)\in X\times Y\} \to X\times Y \quad i_{y_0,\Gamma}(x) = (x,y_0 + \Gammax)$$
a \emph{non-vertical slice} of $X\times Y$.
\end{definition}

Vertical and non-vertical slices induce restriction maps
$$ i_{y_0,\Gamma}^* :  i_{y_0,U}^*J^2(X\times Y) \to J^2(X).$$
$$ i_{x_0}^*: i_{x_0}^*J^2(X\times Y) \to J^2(Y)$$
\end{comment}

\begin{definition}[Products]
Let $X\subset\mathbb R^n$ and $Y\subset \mathbb R^m$ be open, and $F\subset J^2(X)$ and $G\subset J^2(Y)$. Define $$F\#G \subset J^2(X\times Y)$$ by
%$$F\#G = \left \{  ((x,y),r,p,A)\in J^2(X\times Y) : \begin{array}{l}   i_{\Gamma}^*(r,p,A)) \in F_{x}  \text{ and }\\  j^*(r,p,A) \in G_y  \\ \text{for all }  \Gamma\in \Hom(\mathbb R^{n},\mathbb R^{m})\end{array}\right\}.$$
$$(F\#G)_{(x,y)}  = \left \{ \alpha\in J^2_{n+m} : \begin{array}{l}   i_{\Gamma}^*\alpha \in F_{x}  \text{ and }  j^*\alpha \in G_y  \\ \text{for all }  \Gamma\in \Hom(\mathbb R^{n},\mathbb R^{m})\end{array}\right\}.$$
\end{definition}

\begin{lemma}
\begin{enumerate}
\item If $F$ and $G$ are primitive subequations then so is $F\#G$.  Moreover if $F$ and $G$ both have the Negativity Property then so does $F\#G$. 
\item Let $F$ be a constant-coefficient primitive subequation on $X$.  Suppose and $f$ is $F\#\calPos$-subharmonic on some open $\Omega\subset X\times Y$.  The for each $x\in X$ the function $y\mapsto f(x,y)$ is locally convex.
\end{enumerate}
\end{lemma}
\begin{proof}
The reader will easily prove these straight from the definition, or otherwise find the proofs in \cite{MP1}.
\end{proof}

\subsection{The almost everywhere theorem}

We will rely on a very useful theorem of Harvey-Lawson that characterizes $F$-subharmonic semiconvex functions in terms of second order jets almost everywhere.

%\begin{theorem}[Alexandrov's Theorem]\label{thm:alexandrov}
%Suppose that $f: X\to \mathbb R$ is convex.  Then $f$ is twice differentiable for almost all $x_0\in \mathbb R^n$.  
%\end{theorem}

%\begin{definition}[Differentiable Set-Valued Functions]
%We say that a set-valued function $F$ on $\mathbb R^n$ with values subsets $\mathbb R^m$ is \emph{differentiable} at $x_0\in \mathbb R^n$ if there exists a linear $L:\mathbb R^n\to \mathbb R^m$ such that the following holds: for all $\epsilon>0$ there is a $\delta>0$ such that for $\|x-x_0\|<\delta$ and all $u\in F(x)$ and $u_0\in F(x_0)$ we have
%$$ \| u-u_0 - L(x-x_0)|\| \le \epsilon \|x-x_0\|$$
%\end{definition}

%and we will abuse notation in this way going forward.

\begin{definition}[Twice differentiability at a point]
We say that a function $f:X\to \mathbb R$ is \emph{twice differentiable} at $x_0\in X$ if there exists a $p\in \mathbb R^n$ and an $L\in \Sym_n^2$ such that for all $\epsilon>0$ there is a $\delta>0$ such that for $\|x-x_0\|<\delta$  we get
\begin{equation}\label{eq:twicediff} |f(x) - f(x_0) - p.(x-x_0) - \frac{1}{2}  (x-x_0)^tL (x-x_0) | \le \epsilon \|x-x_0\|^2.
\end{equation}
\end{definition}

When $f$ is twice differentiable at $x_0$ then the $p,L$ in \eqref{eq:twicediff} are unique, and moreover in this case $f$ is differentiable at $x_0$ and 
$$p = \nabla f|_{x_0}= \left(  \begin{array}{c} \frac{\partial f}{\partial x_1} \\ \frac{\partial f}{\partial x_2} \\ \vdots \\ \frac{\partial f}{\partial x_n} \end{array}\right)|_{x_0}\in \mathbb R^n.$$
When $f$ is twice differentiable at $x_0$ we shall refer to $L$ as the \emph{Hessian} of $f$ at $x_0$ and denote it by $\Hess(f)|_{x_0}$.  Of course when $f$ is $\mathcal C^{2}$ in a neighbourhood of $x_0$ then $\Hess_{x}(f)$ is the matrix with entries
$$(\Hess(f)_{x_0} )_{ij}: = \frac{\partial^2 f}{\partial x_i\partial x_j}|_{x_0}.$$

\begin{definition}[Second order jet]
Suppose that $f: X\to \mathbb R$ is twice differentiable at $x_0$.    We denote the \emph{second order jet} of $f$ at $x_0$ by
\begin{equation}J^2_{x_0}(f):= (f(x_0), \nabla f|_{x_0}, \Hess(f)|_{x_0}) \in J^2_{n} = \mathbb R\times \mathbb R^n\times \Sym_n^2.\label{eq:secondorderjet}
\end{equation}
%The \emph{reduced second order jet} is
%\begin{equation}J^2_{x_0,red}(f):= (d f|_{x_0}, \Hess(f)|_{x_0}) \in \mathbb R\times \mathbb R^n \times \Sym^2_n.\label{eq:secondorderjetreduced}
%\end{equation}
\end{definition}

We have seen in Alexandrov's Theorem (Theorem \ref{thm:alexandrov:repeat}) that if $f$ is locally semiconvex then $J^2_x(f)$ exists for almost all $x$.

\begin{theorem}[The Almost Everywhere Theorem]\label{thm:ae}
Assume that $F\subset J^2(X)$ is a primitive subequation and let $f:X\to \mathbb R$ be locally semiconvex.    Then
$$ f\in F(X) \Leftrightarrow J^2_{x}(f)\in F_x \text{ for almost all } x\in X.$$
\end{theorem}
\begin{proof}
See \cite[Theorem 4.1]{HL_AE}.
\end{proof}

\section{Partial sup-convolutions}\label{sec:partialsup}
Fix open $U\subset \mathbb R^n$ and $V\subset \mathbb R^m$, and suppose $f:U\times V\to \mathbb R$ is upper-semicontinuous and bounded.

\begin{definition}[Partial-Sup-Convolutions]\label{def:partialsupconvolution} For $\epsilon>0$ the \emph{partial sup-convolution} of $f$ is
\begin{equation}\label{eq:defpartialsupconvolution}
f^{\epsilon,p}(x,y): = \sup_{z\in U} \{ f(z,y) - \frac{1}{2\epsilon} \| z-x\|^2\} \text{ for } (x,y)\in U\times V.
\end{equation}
\end{definition}
For $\delta>0$ let
$$U(\delta) = \{ x\in \mathbb R^n : B_{\delta}(x)\subset U\}.$$

\begin{lemma}[Basic Properties of Partial-Sup-Convolutions]\label{lem:partialsupconvolutionbasicproperties}\ 
%Suppose that $|f|<M$ on $K$ and set
\begin{enumerate}[(i)]
%\item (Slices)  If $i^*_{q,U} f$ is not identically $-\infty$ then
% $$(i^*_{q,U} f)^{\epsilon,p}=i^*_{q,U}( f^{\epsilon,p})$$
\item (Strong Semiconvexity) Assume that for each fixed $x$ the function $y\mapsto f(x,y)$ is convex.  Then 
$$(x,y)\mapsto f^{\epsilon,p}(x,y) + \frac{1}{2\epsilon} \|x\|^2$$ is convex.
%\item (Semiconcavity) Suppose that $i^*_{q,U}f$ is not-identically $-\infty$ and is $\frac{1}{2\epsilon}$-semiconcave.  Then $$i^*_{q,U}(f^{\epsilon,p})\text{ is }\frac{1}{2\epsilon}\text{-semiconcave}.$$
\item (Monotonicity) For $0<\epsilon'\le \epsilon$ we have \begin{equation}\label{eq:orderf}f \le f^{\epsilon',p} \le  f^{\epsilon,p}.\end{equation}
\item Let $\delta: =2  (\epsilon \|f\|_{\infty})^{1/2}$.   Then
%$$ f^{\epsilon,p}(x,y)  = \sup_{z: (z,y)\in K \text{ and } \|z-x\|<\delta} \{ f(z,y) - \frac{1}{2\epsilon} \| z-x\|^2\} \text{ for } (x,y)\in K$$
%and
 $$f^{\epsilon,p}(x,y)  = \sup_{\|\tau\|<\delta} \{ f(x+\tau,y)  - \frac{1}{2\epsilon} \|\tau\|^2\} \text{ for } (x,y)\in U(\delta)\times V.$$
\item (Pointwise convergence) 
$$\lim_{\epsilon\to 0^+} f^{\epsilon,p}(x,y) = f(x,y) \text{ for } (x,y)\in U\times V.$$
\item (Magic-Property)
Suppose that $F$ is a constant-coefficient primitive subequation on $U$ that has the Negativity Property and $f$ is $F\#\calPos$-subharmonic.  Then $f^{\epsilon,p}$ is $F\#\calPos$-subharmonic on $U(\delta)\times V$.
%\item (Preservation of semiconcavity)
%\item (Preservation of argmin)
%\item (Marginal functions)  Let 
%$$g(x;\epsilon): = \inf_y f^{\epsilon,p}(x,y) \text{ and }g(x): = \inf_y f(x,y).$$ 
%Then \begin{equation}\label{eq:ordermarginal}
%g(x)\le g(x;\epsilon')\le g(x;{\epsilon})\text{ for }\epsilon'\le \epsilon.\end{equation} Moreover
%$$\lim_{\epsilon\to 0^+} g(x,{\epsilon}) = g(x)$$
%whenever the limit on the left hand side is finite.
\end{enumerate}
\end{lemma}
\begin{proof}
\begin{align*}f^{\epsilon,p}(x,y) + \frac{1}{2\epsilon} \|x\|^2& = \sup_{z\in U} \{ f(z,y) - \frac{1}{2\epsilon} \|z-x\|^2  + \frac{1}{2\epsilon} \|x\|^2\} \\
&=\sup_{z\in U} \{ f(z,y) + \frac{1}{\epsilon} x.z - \frac{1}{2\epsilon} \|z\|^2\}.
\end{align*}
Now for fixed $z$ the function $y\mapsto f(z,y)$ is assumed to be convex in $y$, and so the function $(x,y) \mapsto f(z,y)$ is convex in $(x,y)$.  Thus, again for $z$ fixed, $(x,y) \mapsto  f(z,y) + \frac{1}{\epsilon} x.z + \frac{1}{2\epsilon} \|z\|^2$ is convex in $(x,y)$, and hence so is $f^{\epsilon,p}(x,y) + \frac{1}{2\epsilon} \|x\|^2$ proving (i).

Item (ii) is immediate.  For (iii) we claim that
\begin{equation}\label{eq:partialsupconvolutionslocalises}f^{\epsilon,p}(x,y) = \sup_{z\in U : \|z-x\|<\delta}\{ f(z,y) - \frac{1}{2\epsilon} \|z-x\|^2\} \text{ for } (x,y)\in U\times V.\end{equation}
To see this let $M: = \|f\|_{\infty}$.  Then for $z\in U$ with  $\|z-x\|\ge \delta = \sqrt{4\epsilon M}$,
$$f(z,y) - \frac{1}{2\epsilon} \|z-x\|^2 \le M - \frac{1}{2\epsilon} \delta^2 = -M \le f(x,y)\le f^{\epsilon,p}(x,y)$$
which proves \eqref{eq:partialsupconvolutionslocalises}.  Then (iii) follows upon making the change of variables $\tau: = z-x$.  For the pointwise convergence fix $(x,y)\in U\times V$ and let $a>f(x,y)$.  Then $f<a$ on some open neighbourhood of $(x,y)$.  Let $\epsilon$ be small enough so that $B_\delta(x)$ is contained in this neighbourhood.  Then \eqref{eq:partialsupconvolutionslocalises} implies $f^{\epsilon,p}(x,y) \le a$, proving (iv).  For the final statement, since $F$ is constant coefficient for any fixed $\tau$ the function $f(x+\tau,y)$ is $F\#\calPos$-subharmonic (where defined), and hence (iii) shows $f^{\epsilon,p}$ as a supremum of $F\#\calPos$-subharmonic functions.  Now being $F\#\calPos$-subharmonic implies that $y\mapsto f(x,y)$ is convex, and so by (i) $f^{\epsilon,p}$ is certainly continuous and hence equal to its upper semicontinuous regularisation.  Thus $f^{\epsilon,p}$ is $F\#\calPos$-subharmonic on $U(\delta)\times V$ as claimed in (v).

%Equation \eqref{eq:ordermarginal} follows immediately from \eqref{eq:orderf}.%  If $\lim_{\epsilon\to 0^+} g(x;\epsilon) = l$ is finite, then for any $\sigma>0$ it holds that $g(x;\epsilon)\le l + \sigma$ for positive $\epsilon$ sufficiently small.  Thus for such $\epsilon$ there is a $y_\epsilon$ with $f^{\epsilon,p}(x,y_\epsilon)<l +2\epsilon$ and hence $g(x)\le l + 2\epsilon$ and so $g(x)\le l$. 
\end{proof}

The next lemma reveals a surprising property of the above construction, namely that the partial sup-convolution of a semiconcave function is fibrewise semiconcave.

\begin{lemma}\label{lem:preservationofconcavity}
Suppose that  $f$ is $\kappa$-semiconcave for some $\kappa>0$.    Then for $\epsilon<\kappa^{-1}$ and fixed $x\in U$ the function
$$ y\mapsto f^{\epsilon,p}(x,y)-\frac{\kappa}{2} \|y\|^2$$
is concave.
\end{lemma}
\begin{proof}
Let $x$ be fixed.  Then
\begin{align*}
 f^{\epsilon,p}(x,y) - \frac{\kappa}{2} \|y\|^2  &=  
 \sup_{z\in U} \{f(z,y) - \frac{\kappa}{2} \|y\|^2 - \frac{1}{2\epsilon} \|z-x\|^2\} \\
&= \sup_{z\in U} \{f(z,y) - \frac{\kappa}{2} \|z\|^2 - \frac{\kappa}{2} \|y\|^2  +\frac{\kappa-\epsilon^{-1}}{2} \|z\|^2  + \frac{1}{\epsilon} z.x - \frac{1}{2\epsilon} \|x\|^2\}.
\end{align*}
Observe that (since $x$ is fixed and $\kappa-\epsilon^{-1}<0$) the function $(z,y)\mapsto \frac{\kappa-\epsilon^{-1}}{2} \|z\|^2  + \frac{1}{\epsilon} z.x - \frac{1}{2\epsilon} \|x\|^2$ is convex as a function of $(z,y)$.  Furthermore by hypothesis $f(z,y) - \frac{\kappa}{2} \|z\|^2 - \frac{\kappa}{2} \|y\|^2$ is concave in $(z,y)$.  Hence  $y\mapsto f^{\epsilon,p}(x,y) - \frac{\kappa}{2} \|y\|^2$ is a supremum of functions concave in two variables, and thus is concave.
\end{proof}

\section{$F$-subharmonicity of marginal functions}\label{sec:fsubharmonicityofmarginals}
Let $\Omega\subset \mathbb R^{n+m}$ be open, convex and such that $\Omega_x$ is connected for all $x$.

\begin{proposition}\label{prop:semiI}
Let  $F\subset J^2(\mathbb R^{n})$ be a primitive subequation.  Let $f:\Omega\to \mathbb R$ be $F\#\calPos$-subharmonic, and suppose that for some $\sigma,\kappa_1,\kappa_2>0$ the function
\begin{equation}\label{eq:semiIH1} f(x,y) + \frac{\kappa_1}{2} \|x\|^2-\frac{\sigma}{2} \|y\|^2 \text{ is convex}
\text{ and }\end{equation}
and for each fixed $x$ the function
\begin{equation} \label{eq:semiIH2}  y\mapsto f(x,y) - \frac{\kappa_2}{2} \|y\|^2 \text{ is concave}\end{equation}
and that $\gamma(x)=\argmin_f(x)$ is single valued.  Then
$$ g(x): = \inf_{y\in \Omega_x} f(x,y)$$
is $F$-subharmonic.
\end{proposition}

%\begin{remark}
%In Theorem \ref{thm:semiII} we will improve this result by relaxing the hypothesis \eqref{eq:semiIH1}.
%\end{remark}

\begin{proof}
By hypothesis $$g(x) = f(x,\gamma(x)).$$
Now $g$ is $\kappa$-semiconvex (Lemma \ref{lem:basicconvexityII}) so by Alexandrov's Theorem (Theorem \ref{thm:alexandrov:repeat}) $g$ is twice differentiable almost everywhere.  Furthermore \eqref{eq:semiIH1} allows us to invoke our results on the argmin function, so by Corollary \ref{cor:wisdiffae} $\gamma$ is differentiable almost everywhere.  Let $x_0$ be a point where $g$ is twice differentiable and $\gamma$ is differentiable, and we will show
\begin{equation}J^2_{x_0} g = (g(x_0), \nabla g|_{x_0}, \Hess_{x_0}(g)) \in F_{x_0}\label{eq:J2g}.\end{equation}
By the Almost Everywhere Theorem (Theorem \ref{thm:ae}) this implies that $g$ is $F$-subharmonic.

Actually we will show that for any $\epsilon>0$ it holds that
\begin{equation}(g(x_0), \nabla g|_{x_0}, \Hess_{x_0}(g)+\epsilon \Id_{n}) \in F_{x_0}.\label{eq:j2gepsilon}\end{equation}
Letting $\epsilon\to 0$ and using that $F_{x_0}$ is closed yields \eqref{eq:J2g}.

To this end set $y_0: = \gamma(x_0)$ and
$$\Gamma:= D\gamma|_{x_0}\in \Hom(\mathbb R^{n},\mathbb R^{m})$$
and $d(x,y)$ be the vertical distance between $(x,y)\in \mathbb R^{n}\times \mathbb R^{m}$ and the tangent to the graph of $\gamma$ at $(x_0,y_0)$, so  
$$ d(x,y) := \|y - y_0 - \Gamma (x-x_0)\| \text{ for } (x,y)\in \mathbb R^{n}\times \mathbb R^{m}.$$
Consider the quadratic
$$ q(x,y) = g(x_0) + \nabla g|_{x_0}.(x-x_0) +\frac{1}{2}  (x-x_0)^t\Hess_{x_0}(g) (x-x_0)+ \frac{\epsilon}{2}\|x-x_0\|^2 + \kappa_2  d(x,y)^2$$
for 
 $(x,y)\in \mathbb R^{n}\times \mathbb R^{m}$.  By construction $$q(x_0,y_0) = g(x_0) = f(x_0,\gamma(x_0)) = f(x_0,y_0),$$  and in Lemma \ref{lem:semiI} below we show that $q\ge f$ sufficiently near $(x_0,y_0)$.
Hence $(x_0,y_0)$ is an upper contact point for $f$ and
 \begin{align}J^2_{(x_0,y_0)}(q) &=  \left( q(x_0,y_0), \nabla q|_{(x_0,y_0}, \Hess_{(x_0,y_0}(q)\right)\\
 &=\left(f(x_0,y_0),\left(\begin{array}{c} \nabla g|_{x_0} \\0 \end{array}\right),\left(\begin{array}{cc} \Hess_{x_0}(g) +\epsilon \Id_{n}  + 2\kappa_2\Gamma^t\Gamma& -2\kappa_2\Gamma^t \\ -2\kappa_2\Gamma & 2\kappa_2\Id_{m}\end{array}\right)\right)
 \end{align}
 is an upper-contact jet of $f$ at $(x_0,y_0)$.     So as $f$ is $F\#\calPos$-subharmonic we have
$$ J^2_{(x_0,y_0)}(q)  \in (F\#\calPos)_{(x_0,y_0)}.$$
And from the definition of $i_{\Gamma}^*$,
 $$ i_{\Gamma}^* (J^2_{(x_0,y_0)}(q))=\Hess_{x_0}(g) +\epsilon \Id_{n}$$
which must lie in $F_{x_0}$.  This gives \eqref{eq:j2gepsilon} and completes the proof.
\end{proof}

 \begin{lemma}\label{lem:semiI}
 With the notation as in the proof of Theorem \ref{prop:semiI} we have 
  \begin{equation}\label{eq:qlarger} q(x,y) \ge f(x,y) \text{ for } (x,y) \text{ sufficiently near } (x_0,y_0).\end{equation}
 \end{lemma}
 \begin{proof}
Fix $\epsilon'>0$ small enough so $\epsilon' + \kappa_2\epsilon'^2 <\epsilon/2$.  That $\Gamma=D\gamma|_{x_0}$ means there is a $\delta>0$ such that for all $\|x-x_0\|<\delta$
 $$ \|\gamma(x) - y_0 - \Gamma(x-x_0)\| \le \epsilon' \|x-x_0\|.$$
 Shrinking $\delta$ is necessary, the definition of $g$ being twice differentiable at $x_0$ means \eqref{eq:twicediff} that for $\|x-x_0\|<\delta$ we also have
 $$ | g(x) - g(x_0) - \nabla g|_{x_0}.(x-x_0) -\frac{1}{2} (x-x_0)^t \Hess_{x_0}g(x-x_0) | \le \epsilon' \|x-x_0\|^2.$$
 
 Consider now a point $(x,y)$ with $\|x-x_0\|<\delta$ and $\|y-y_0\|<\delta$.  Then
  \begin{align}
  \| y-\gamma(x)\| &\le \|y-y_0 - \Gamma (x-x_0) \| + \|y_0+\Gamma (x-x_0)-\gamma (x)\|\\
  &\le d(x,y) + \epsilon' \|x-x_0\|. \end{align}
So
$$\| y-\gamma (x)\|^2 \le 2\epsilon'^2 \|x-x_0\|^2 + 2d(x,y)^2.$$
Now we use (in an essential way) hypothesis \eqref{eq:semiIH2}.  Since $\gamma (x)$ is the minimum of the function $y'\mapsto f(x,y')$ \eqref{eq:semiIH2} implies
$$ f(x,y) \le f(x,\gamma(x)) + \frac{\kappa_2}{2} \|y-\gamma(x)\|^2.$$
Thus
\begin{align*}
f(x,y) &\le g(x) + \kappa_2(\epsilon'^2 \|x-x_0\|^2 + d(x,y)^2)\\
&\le g(x_0) + \nabla g|_{x_0} (x-x_0) + \frac{1}{2} (x-x_0)^t\Hess_{x_0}g(x-x_0) \\&\quad + (\kappa_2\epsilon'^2 + \epsilon')\|x-x_0\|^2   + \kappa_2 d(x,y)^2\\
&\le q(x,y)
\end{align*}
as required.
\end{proof}

\begin{proof}[Proof of Theorem \ref{thm:introminimum}]
Let $f:X\times \mathbb R\to \mathbb R$ be locally semiconcave, bounded from below and $F\#\calPos$-subharmonic.  We are to show that $g(x):= \inf_y f(x,y)$ is $F$-subharmonic.

We first claim that without loss of generality we may assume in addition that for each $x$ it holds that $\argmin_f(x)$ is non-empty and single valued.  To prove this, for $j\ge 1$ let
$$ f_j(x,y) = f(x,y) + \frac{1}{j} \|y\|^2.$$
As $F$ depends only on the Hessian part, $f_j$ is still $F\#\calPos$-subharmonic, and is still bounded from below and semiconcave.  Moreover since $f$ is bounded from below, for each fixed $x$ the function $y\mapsto f(x,y)$ is strictly convex and tends to infinity as $|y|$ tends to infinity, implying that it has a unique global minimum.   By assumption the theorem applies to $f_j$ meaning that letting $g_j(x):= \inf_y f_j(x,y)$ the function $g_j$ is $F$-subharmonic.  But $g_j\searrow g$ pointwise as $j\to \infty$, and thus $g$ will be $F$-subharmonic as well, proving the claim.\medskip

So from now on assume that $\gamma(x)=\argmin_f(x)$ is single valued.  Fix $x_0\in \mathbb R^n$.  As $\gamma$ is continuous, there exist small balls $x_0\in U\subset X$ and $\gamma(x_0)\in V\subset \mathbb R^m$ such that $\gamma(U)\subset V$ and $f$ is semiconcave on $U\times V$.   For $\epsilon>0$ consider the function
$$f_\epsilon(x,y) := f^{\epsilon,p}(x,y) + \frac{\epsilon}{2} \|y\|^2 = \sup_{z\in U} \{ f(z,y) - \frac{1}{2\epsilon} \|z-x\|^2\} + \frac{\epsilon}{2} \|y\|^2.$$
We claim that for $\epsilon$ sufficiently small the following all hold:
\begin{enumerate}[(i)]
\item $f_{\epsilon}(x,y) + \frac{1}{2\epsilon} \|x\|^2  - \frac{\epsilon}{2} \|y\|^2$ is convex.
\item $f_{\epsilon}$ is $F\#\calPos$-subharmonic on $U'\times V$ for some smaller ball $x_0\in U'\subset U$.
\item $f^{\epsilon}\searrow f$ pointwise on $U\times V$ as $\epsilon\to 0^+$.
\item There is a $\kappa_2>0$ such that for each $x\in U$ the function $y\mapsto f_{\epsilon}(x,y) - \frac{\kappa_2}{2} \|y\|^2$ is concave.
\end{enumerate}
Items (i,ii,iii) follow from Lemma \ref{lem:partialsupconvolutionbasicproperties} (we have used here the hypothesis that $F$ depends only on the Hessian part so adding a multiple of $\|y\|^2$ preserves the property of being $F\#\calPos$-subharmonic by Lemma \ref{lem:sumsubharmoniconvex}).    The statement (iv) follows from Lemma \ref{lem:preservationofconcavity} (observing that the addition of $\frac{\epsilon}{2}\|y\|^2$ to the partial sup-convolution only means we may need to increase the value of $\kappa_2$)

Thus we are in a position to apply Proposition \ref{prop:semiI} to $f_{\epsilon}$ to conclude that if
$$g_{\epsilon}(x): = \inf_{y\in V} f_{\epsilon}(x,y)$$
then $g_{\epsilon}$ is $F$-subharmonic on $U'$.  But by (iii) if $x\in U'$ then
$$ g_{\epsilon}(x) \searrow \inf_{y\in V} f(x,y) = f(x,\gamma(x)) = g(x) \text{ as } \epsilon\to 0^+$$
and hence $g$ is also $F$-subharmonic on $U'$.  Since $x_0$ was arbitrary we conclude $g$ is $F$-subharmonic on all of $\mathbb R^n$ as required.
\end{proof}

\appendix
\section{The Implicit Function Theorem for Lipschitz Functions}\label{sec:appendiximplicit}

The following version of the Implicit function theorem is taken from \cite[Theorem 5.1]{Wuertz}, and we include a proof for convenience.
%m->r 
%n->s

\begin{theorem}[Lipschitz Implicit Function Theorem]
Let $U_1\subset \mathbb R^{r}$ and $U_2\subset \mathbb R^{s}$ be open and
$$ J : U_1\times U_2 \to \mathbb R^{s}$$
be Lipschitz with the property that there is a $K>0$ such that 
$$\|J(p,y_1) - J(p,y_2)\| \ge K \|y_1-y_2\| \text{ for all } (p,y_1), (p,y_2) \in U_1\times U_2.$$
Suppose $a\in U_1, b\in U_2$ is such that
$$ J(a,b)=0.$$
There there exists an open $a\in V\subset U_1$ and a Lipschitz map 
$$\phi:V\to U_2$$
such that $\phi(a)=b$ and
\begin{equation}J(p,\phi(p)) =0 \text{ for all } p\in V.\label{eq:implicitJ}\end{equation}
\end{theorem}

\begin{proof}
For small $\epsilon>0$ (to be determined) let
$$ \hat{J}: U_1\times U_2\to \mathbb R^{r+s} \text{ be } \hat{J}(p,y)  =(p, \epsilon J(p,y))$$
which is Lipschitz as $J$ is assumed to be Lipschitz.   We claim that as long as $\epsilon$ is sufficiently small, $\hat{J}$ is bi-Lipschitz, i.e. there is a $C>0$ such that
\begin{equation}\|\hat{J}(p_1,y_1) - \hat{J}(p_2,y_2)\| \ge C \| (p_1,y_1)  - (p_2,y_2)\| \label{eq:hatJbilipschitz}\end{equation}
for all $(p_i,y_i)\in U_1\times U_2$.

To see this, say $J$ has Lipschitz constant $M$ and let $(p_i,y_i)\in U_1\times U_2$.  Then
\begin{align*} K^2 \|y_1-y_2\|^2 &\le \|{J}(p_1,y_1) - {J}(p_1,y_2)\|^2 \\
&\le 2 ( \|{J}(p_1,y_1) - {J}(p_2,y_2)\|^2 + \|{J}(p_2,y_2) - {J}(p_1,y_2)\|^2)\\
&\le 2 \|{J}(p_1,y_1) - {J}(p_2,y_2)\|^2  + 2M^2 \|p_2-p_1\|^2.
\end{align*}
Multiplying by $\epsilon^2/2$ and rearranging gives
\begin{align*}
\frac{K^2\epsilon^2}{2} \|y_1-y_2\|^2 + (1-\epsilon^2 M^2) \|p_1-p_2\|^2& \le \epsilon^2 \|{J}(p_1,y_1) - {J}(p_2,y_2)\|^2 + \|p_1-p_2\|^2\\
&= \| \hat{J}(p_1,y_1) - \hat{J}(p_2,y_2)\|^2.
\end{align*}
So if we take $\epsilon$ small enough so $1-\epsilon^2 M^2 \ge \frac{K^2\epsilon^2}{2}=:C^2$ then
$$ \| \hat{J}(p_1,y_1) - \hat{J}(p_2,y_2)\|^2 \ge C^2( \|y_1-y_2\|^2 + \|p_1-p_2\|^2) = C^2 \|(p_1,y_1) - (p_2,y_2)\|^2$$
as claimed in \eqref{eq:hatJbilipschitz}.

In particular $\hat{J}$ is continuous and injective.  Thus by Brouwer's Invariance of Domain Theorem \cite[Corollary 19.8]{Bredon}, $\hat{J}$ is an open map.  So $V:= \hat{J}(U_1\times U_2)\subset \mathbb R^{r+s}$ is open and $\hat{J}:U_1\times U_2\to V$ is a continuous bijection with continuous inverse $\hat{J}^{-1}:V\to U_1\times U_2$.   In fact as $\hat{J}$ is bi-Lipschitz, we get that $\hat{J}^{-1}$ is Lipschitz.

Denote by $\pi_1:\mathbb R^r\times \mathbb R^s\to \mathbb R^r$ and $\pi_2:\mathbb R^r\times \mathbb R^s\to \mathbb R^s$ the projections, and let $B$ be a small ball around $a$ so that $B\subset U_1$ and $B\times \{0\} \subset \hat{J}(\pi_2^{-1}(U_2))$.   Define $\phi:B\to U_1\subset \mathbb R^r$ by
$$\phi(p) = \pi_2 \hat{J}^{-1}(p,0).$$
 Then $\phi$ is Lipschitz and $\hat{J}^{-1}(a,0)=(a,b)$ gives $\phi(a)=b$.  Moreover if $p\in V$ then
\begin{align*}
(p,0) &= \hat{J} \hat{J}^{-1} (p,0)  = J ( \pi_1 \hat{J}^{-1}(p,0), \pi_2 \hat{J}^{-1}(p,0)) 
\\& =J ( \pi_1 \hat{J}^{-1}(p,0), \phi(p))  = ( \pi_1 \hat{J}^{-1}(p,0), \epsilon J( \pi_1 \hat{J}^{-1}(p,0) ,\phi(p)).
\end{align*}
Thus
$$p =\pi_1 \hat{J}^{-1}(p,0)$$
and
$$0 = \epsilon J( \pi_1 \hat{J}^{-1}(p,0) ,\phi(p))  = \epsilon J(p,\phi(p))$$
proving \eqref{eq:implicitJ}

\end{proof}

\bibliography{minimumprinciple2}{}
\bibliographystyle{plain}
\end{document}